\documentclass[12pt]{amsart}
\usepackage{graphicx, fullpage}
\usepackage{amsmath,amssymb,amsthm,amscd, bm}
\usepackage{fancyhdr}
\usepackage[mathscr]{eucal}
\usepackage{amsfonts}
\usepackage[T1]{fontenc}
\usepackage{xspace}
\usepackage{esint}
\theoremstyle{plain}
\newtheorem {thm}{Theorem}[section]
\newtheorem {lem}[thm]{Lemma}

\newtheorem {prop}[thm]{Proposition}
\numberwithin{equation}{section}

\newcommand{\bC}{\mathbb C}

\newcommand{\bP}{\mathbb P}
\newcommand{\bN}{\mathbb N}

\newcommand{\bR}{\mathbb R}

\newcommand{\var}{\varepsilon}

\newcommand{\GL}{\text{GL}}

\newcommand{\re}{\text{Re}\,}

\newcommand{\sff}{\mathsf f}
\newcommand{\sgg}{\mathsf g}
\newcommand{\sh}{\mathsf h}
\newcommand{\sk}{\mathsf k}

\newcommand{\sH}{\mathsf H}
\newcommand{\sK}{\mathsf K}

\begin{document}
\title{\bf On a construction of hermitian metrics on holomorphic vector bundles}
\author{L\'aszl\'o Lempert}
\address{Department of  Mathematics,
Purdue University, 150N University Street, West Lafayette, IN
47907-2067, USA}

%\thispagestyle{empty}
%\end{titlepage}
\abstract
With any holomorphic vector bundle $E$ over a compact base one can associate a line bundle, denoted
$O_{\bP E}(1)$. According to a conjecture of Griffiths, if $O_{\bP E}(1)$ admits a positively curved hermitian
metric $k$, then $E$ also admits a positively curved hermitian metric, $h$. In this paper we show that, while
the conjecture may be correct, it is not possible to obtain $h$ out of $k$ by a fiberwise construction that is 
functorial.
\endabstract
\maketitle
\section{Introduction}    % Section 1

The ``construction'' in the title refers to one motivated by a conjecture of Griffiths. 

With any holomorphic vector bundle $E$ over a complex manifold $S$
one can associate a holomorphic line bundle $L\to\bP E$. In our notation the base 
$\bP E=\coprod_{s\in S}\bP E_s$ is the set of
subspaces of codimension one (``hyperplanes'', in what follows) in various fibers $E_s$, on which the local
trivializations of $E$ induce the structure of a complex manifold, a locally trivial holomorphic fiber bundle
$\pi:\bP E\to S$. The fiber of $L$ over $P\in\bP E_s$ 
is $E_s/P$. More precisely, the pullback bundle
$\pi^*E\to\bP E$ contains a tautological subbundle $H$ of corank $1$, whose fiber over ${P}\in\bP E_s$---a hyperplane in $ E_s$---is (the
pullback of) ${P}\subset E_s$, and $L=(\pi^*E)/H$. Our notation for $L$ is $O_{\bP E}(1)$; its restriction
to $\bP E_s$ is the hyperplane section bundle $O_{\bP E_s}(1)$.

One way to formulate Griffiths' conjecture in \cite{G69} is, assuming $S$ compact, that if $E$ is ample as defined by
Hartshorne \cite{H66}, then it is positive. Another, equivalent way, the one we will work with, is that if 
$O_{\bP E}(1)\to\bP E$ admits a smooth hermitian metric of positive curvature, then $E$ also admits one. (The
converse implication is easy to show.) Proofs have been attempted repeatedly, but for $\dim S\ge 2$ 
the conjecture is still
open. The case of $S$ a Riemann surface has been settled by Umemura \cite{U73}, though.

An appealing approach to the conjecture is to try to produce a positively curved hermitian metric $h$ on $E$ out of a
positively curved hermitian metric $k$ on $O_{\bP E}(1)$ by a {\sl fiberwise} construction. 
What this means is, roughly, that with each 
complex vector space $V$ of dimension $r=\text{rk}\, E$ and each positively curved hermitian metric $\sk$ on
$O_{\bP V}(1)$ one should associate a hermitian metric $\sh$ on $V$. Applied to the fibers $E_s=V$ this then
associates  with $k$ a hermitian metric $h$ on $E$. The hope is that if the fiberwise construction $\sk\mapsto\sh$ 
is cleverly chosen, $h$ will be smooth and positively curved. A strong result by Berndtsson on Griffiths' conjecture 
was obtained by a similar fiberwise construction, \cite[Theorem 7.1] {B09}.
Out of a positively curved hermitian
metric $\sk$ on $O_{\bP V}(1)$ Berndtsson constructs a Hermitian metric not on $V$ but on $V\otimes\det V$. 
Accordingly, out of $k$ on $O_{\bP E}(1)$ he obtains a hermitian metric $h$ on $E\otimes\det E$, which he proves 
to have positive curvature if $k$ had positive curvature.\footnote{Assuming $S$ is projective, Mourougane and Takayama
obtained the same result in \cite{MT07} by a somewhat related construction, which, however, is not entirely fiberwise.} 
In a similar vein, in \cite{L26a} we associated a
hermitian metric on $V$ itself with a hermitian metric on $O_{\bP V}(1)$ by solving a variational problem; but  
we found that the resulting construction of a hermitian metric $h$ on $E$ out of a positively curved hermitian metric
on $O_{\bP E}(1)$ does not always produce even a semipositively curved $h$ \cite[section 10]{L26a}.

In this paper we show that this failure is not accidental, and all fiberwise constructions, as long as they are functorial,
of a hermitian metric $h$ on $E$ out of positively curved $k$ will some of the time produce an $h$ whose curvature fails to be even semipositive.
We formulate a precise theorem in the next section.

Kuang--Ru Wu made a number of comments and suggestions on the first draft of this paper, many of which have
been incorporated in the current version.

\section{Terminology, notation, and the theorem} %section 2

If $V$ is a finite dimensional complex vector space, by a hermitian metric on $V$ we mean a function 
$\sh:V\to\bR$ that can be written with a positive definite hermitian form $\hat h:V\times V\to\bC$ as $\sh(v)=\hat h(v,v)$.
By a hermitian metric, or just a metric, on a holomorphic vector bundle $E\to S$ we mean a function
$h:E\to\bR$ that restricts to each fiber as a hermitian metric in the sense above. If $h$ is $C^\infty$, we talk
about a smooth metric. The curvature of the metric $h$ is seminegative if $h$ or, equivalently $\log h$, is 
plurisubharmonic on $E$. The curvature is negative if $h$ is smooth and strongly plurisubharmonic away
from the zero section
($i\partial\bar\partial h>0$). The curvature of $h$ is (semi)positive if the dual metric $h^*$ on the dual bundle
$E^*\to S$ has (semi)negative curvature. Often it is convenient to say positively etc. curved instead of having positive etc.
curvature.

We denote by $\sH_V$ the set of hermitian metrics on the vector space $V$, and by $\sK_V$ the set of positively
curved metrics on $O_{\bP V}(1)$. An isomorphism $\phi:V'\to V$ of vector spaces induces a map 
$\phi_\sH:\sH_V\to\sH_{V'}$ by pullback: $\phi_\sH\sh=\sh\circ\phi$, $\sh\in\sH_V$. There is also a
map $\phi_\sK:\sK_V\to\sK_{V'}$ obtained as follows: $\phi$ induces isomorphisms
$V'/{P}\to V/\phi{P}$ of the quotients, ${P}\in\bP V'$, which together
define an isomorphism $O_{\bP V'}(1)\to O_{\bP V}(1)$; the pullback of $\sk\in\sK_V$ along this isomorphism is
then $\phi_\sK \sk$.

Self isomorphisms $V\to V$ form a group $\GL(V)$, and 
$\phi_\sH,\phi_\sK$ for $\phi\in\GL(V)$ define left actions of $\GL(V)$ on $\sH_V$, $\sK_V$. 
Suppose we are given
a $\GL(\bC^r)$ equivariant map $\Psi:\sK_{\bC^r}\to\sH_{\bC^r}$, i.e.,
\[
\phi_{\sH}\circ\Psi=\Psi\circ\phi_{\sK},\qquad \phi\in\GL(\bC^r).
\]
Equivariance has the effect that the association $\sk\mapsto\Psi\sk$ depends only on the vector space structure of
$\bC^r$ and not on the choice of coordinates. Accordingly,
our $\Psi$ can be used to construct $\GL(V)$ equivariant maps $\Psi_V:\sK_V\to\sH_V$ for all $r$ dimensional
vector spaces $V$, by choosing an isomorphism $\psi:\bC^r\to V$, and letting 
$\Psi_V=\psi_\sH^{-1}\circ\Psi\circ\psi_\sK$. Since $\Psi$ was equivariant, it does not matter which isomorphism
$\psi$ we use here. If $E\to S$ is now a holomorphic vector bundle of rank $r$, and $k$ is a metric on
$O_{\bP E}(1)$ whose restrictions to $O_{\bP E_s}(1)$ have positive curvature, we define a metric $h$ on $E$ by 
\[
h|E_s=\Psi_{E_s}\big(k|O_{\bP E_s}(1)\big),\qquad s\in S,
\]
that we denote $h=\Psi k$.

\begin{thm} %2.1
Let $S$ be a positive dimensional projective algebraic manifold and $r=2,3,\dots$. There are a holomorphic vector bundle $E\to S$ of rank $r$ and a positively curved metric $k$ on $O_{\bP E}(1)$ such that for no equivariant 
$\Psi:\sK_{\bC^r}\to\sH_{\bC^r}$ has the metric $\Psi k$ on $E$ semipositive curvature.
\end{thm}

The result is analogous to \cite[Theorem 1.4]{L26b}, that is about associating hermitian ellipsoids with convex domains in
$\bC^n$. The main ideas of the proofs are similar.

Theorem 2.1 should not be taken as an indication that Griffiths' conjecture might be wrong. The bundle $E$ of the 
theorem is $\Lambda^{\oplus r}$, with $\Lambda\to S$ a positive line bundle, and definitely admits positively
curved metrics $h$. These $h$ may even be of form $\Psi\kappa$ for some positively curved metric $\kappa$
on $O_{\bP E}(1)$.---As Wu points out, they will necessarily be, for a certain, natural, choice  of $\Psi$ 
\cite{W26}.---All that the theorem brings to Griffiths' conjecture is that the conjecture cannot be proved
by taking an arbitrary positively curved metric $k$ on $O_{\bP E}(1)$, and passing to the metric $\Psi k$
with some cleverly chosen $\Psi$.\footnote{ It is worthwhile to compare this impossibility with Wu's proof of
\cite[Theorem 4]{W23}, by a functorial fiberwise construction of a positively curved metric $h$ on $E$, out of a 
positively curved $k$ on $O_{\bP E}(1)$ 
and a further hermitian metric on $E\otimes\det E^*$, satisfying
a curvature constraint.}

These considerations lead to the questions: Given an equivariant $\Psi:\sK_{\bC^r}\to\sH_{\bC^r}$, how are
the curvatures of $k$ and $\Psi k$ related? What property of $k$ will imply that $\Psi k$ has at least semipositive
curvature? At this point we have no meaningful answers.

\section{Basic constructions} %section 3

The material collected here is not new.---For any vector bundle $F\to T$ we will identify its zero section with the 
base $T$.

If $V$ is a finite dimensional complex vector space, there is a holomorphic map
$\pi_V:O_{\bP V}(-1)=O_{\bP V}(1)^*\to V^*$ defined as follows. Let ${P}\subset V$ be a hyperplane and 
$\xi\in O_{\bP V}(-1)_{P}$. Thus $\xi$ is a linear form $O_{\bP V}(1)_{P}=V/{P}\to\bC$. Composition
with projection $V\to V/{P}$ yields a linear form $\pi_V(\xi):V\to\bC$, vanishing on ${P}$, i.e., 
an element of $V^*$. The
restriction of $\pi_V$ to the complement of the zero section is a biholomorphism 
$O_{\bP V}(-1)\setminus\bP V\to V^*\setminus\{0\}$.

We now move to the construction of metrics on $O_{\bP V}(\pm1)$.  Consider a function $\sff:V\to[0,\infty)$ that
is convex, absolutely $2$-homogeneous in the sense that 
\[
\sff(\lambda v)=|\lambda|^2\sff(v),\qquad \lambda\in\bC,\quad v\in V;
\]
away from $0$ is positive, smooth, and strongly convex, { i.e., its Hessian is positive definite}. 
Thus $\sqrt\sff$ is a norm on $V$, smooth away from $0$. An element of
the fiber $O_{\bP V}(1)_{P}=V/{P}$ is a translate $v+{P}\subset V$ of ${P}$ ($v\in V$), and
\begin{equation} %3.1
\sk_\sff(v+{P})=\min_{v+{P}} \sff
\end{equation}
defines a metric on $O_{\bP V}(1)$. Next, denote by $\langle\, , \rangle:V^*\times V\to\bC$ the duality pairing.
The Legendre transform $\sgg:V^*\to[0,\infty)$ of $\sff$ is
\begin{equation} %3.2
\sgg(\beta)=\sup_{v\in V}\re\langle\beta,v\rangle-\sff(v),\qquad\beta\in V^*.
\end{equation}
It is smooth, positive, and strongly convex on $V^*\setminus\{0\}$, and absolutely $2$-homogeneous, for 
\[
|\lambda|^2\sgg(\beta)=|\lambda|^2\sup_{v\in V}\big(\re\langle\beta,v/\bar\lambda\rangle-\sff(v/\bar\lambda)\big)
=\sup_{v\in V}\big(\re\langle\lambda\beta, v\rangle-\sff(v)\big)=\sgg(\lambda\beta),
\]
at least if $\lambda\neq0$; but the end result holds for $\lambda=0$, too.
Thus $\sqrt\sgg$ is the norm on $V^*$, dual to $\sqrt\sff$; if $\sff$ is a hermitian metric on $V$, then $\sgg$ is the 
dual hermitian metric. The pullback of $\sgg$ by $\pi_V$ defines a smooth metric 
$\sgg\circ\pi_V:O_{\bP V}(-1)\to[0,\infty)$.

\begin{prop} %3.1
The metrics $\sk_\sff$ and $\sgg\circ\pi_V$ are duals of one another.
\end{prop}
\begin{proof}
Let ${P}\in\bP V$ and $\xi\in O_{\bP V}(-1)_{P}$ a nonzero vector, 
$\beta=\pi_V(\xi)\in V^*\setminus\{0\}$. By the construction of $\pi_V$, the kernel of $\beta$ is ${P}$.
Fix $v\in V$ so that
$\langle\beta,v\rangle=1$. 
Since $V=\bigcup_{\lambda\in\bC}\lambda v+P$, 
%(3.2) gives
\begin{align*}
(\sgg\circ\pi_V)(\xi)&=\sgg(\beta)=
\sup\big\{\re\langle\beta,w\rangle-\sff(w):\lambda\in\bC,\,w\in \lambda v+{P}\big\}\\
&=\sup_{\lambda\in\bC} (\re\lambda-\min_{\lambda v+{P}} \sff)
=\sup_{\lambda\in\bC}\big(\re\lambda-\sk_\sff(\lambda v+{P})\big).
\end{align*}
The last expression is the value at $\xi$ of the Legendre transform of the restriction $\sk_\sff|O_{\bP V}(1)_{P}$. 
The restriction is
a hermitian metric and its Legendre transform is its dual; this proves the claim.
\end{proof}

\begin{prop} %3.2
$\sgg\circ\pi_V$ is negatively and, equivalently by Proposition 3.1, $\sk_\sff$ is positively curved.
\end{prop}

This is a special case of Proposition 3.3 below.---In section 2 we associated with any isomorphism $\phi:V'\to V$ 
a map $\phi_\sK:\sK_V\to\sK_{V'}$. It is immediate that
\begin{equation} %3.3
\phi_\sK\sk_\sff=\sk_{\sff\circ\phi}.
\end{equation}

If $E\to S$ is a holomorphic vector bundle, the above constructions can be done fiberwise. We first obtain a
holomorphic map $\pi_E:O_{\bP E}(-1)\to E^*$, biholomorphic between $O_{\bP E}(-1)\setminus\bP E$ and
$E^*\setminus S$. Second, suppose $f:E\to[0,\infty)$ is convex and absolutely $2$-homogeneous on the fibers,
and away from the zero section it is smooth, positive, and fiberwise strongly convex. Often such $f$ are
called (convex) Finsler metrics. Applying (3.1) to each restriction $f|E_s$ gives rise to a metric on $O_{\bP E}(1)$,
that we denote $k_f$.

\begin{prop} %3.3
The fiberwise Legendre transform $g:E^*\to[0,\infty)$ of $f$ is plurisubharmonic if and only if $k_f$ is semipositively
curved; $g$ is smooth and strongly plurisubharmonic away from the zero section
if and only if $k_f$ is positively curved.
\end{prop} 

A dual version of the proposition was proved by Kobayashi \cite[Theorem 4.1]{K75}. 
\begin{proof}
By Proposition 3.1 $k_f^*=g\circ\pi_E:O_{\bP E}(-1)\to[0,\infty)$. This latter is plurisubharmonic if and
only if $g$ is; it is smooth and strongly plurisubharmonic away from $\bP E$ if and only if away from the
zero section $g$ is.
\end{proof}

Finally, consider a holomorphic isomorphism $\phi:E'\to E$ of vector bundles $E'\to S'$, $E\to S$. Thus there is
a biholomorphism $b:S'\to S$ and the restrictions of $\phi$ produce isomorphisms $\phi_s:E'_s\to E_{b(s)}$ of
vector spaces for all $s\in S'$. Out of metrics $h$ on $E$ and $k$ on $O_{\bP E}(1)$ one can construct metrics
$h'$ on $E'$ and $k'$ on $O_{\bP E'}(1)$ by applying $(\phi_s)_\sH$, respectively, $(\phi_s)_\sK$ on each fiber:
\[
h'|E'_{s}=(\phi_s)_\sH(h|E_{b(s)}),\qquad k'|O_{\bP E'_{s}}(1)=(\phi_s)_\sK\big(k|O_{\bP E_{b(s)}}(1)\big).
\]
Then $(E',h')$ and $(E,h)$ on the one hand, and $\big(O_{\bP E'}(1),k'\big)$ and $\big(O_{\bP E}(1), k\big)$ 
on the other are isometrically isomorphic. Further down we will use the notation
\begin{equation} %3.4
h'=\phi_\sH h,\qquad k'=\phi_\sK k.
\end{equation}

\section{A variant of Theorem 2.1} %Section 4

Most of the construction underlying Theorem 2.1 will involve not a bundle over a projective manifold but a trivial vector
bundle $F=B\times V\to B$ of rank $r$ over the unit ball $B$ in some Euclidean space $\bC^m$, 
$m=1,2,\dots$ and $r=2,3,\dots$ arbitrarily fixed. We write  $||\,\,||$ for the Euclidean norm on $\bC^m$.

\begin{thm} %4.1
There are a positively curved metric $k$ on $O_{\bP F}(1)$ and positive numbers $\rho<\rho'<1$ such that
\newline\phantom{if }
(a) with any $\GL(\bC^r )$ equivariant map
$\Psi:\sK_{\bC^r}\to\sH_{\bC^r}$ the metric $\Psi k$ on $F|\rho B$ fails to have semipositive curvature; %and
\newline\phantom{if } 
(b) over $B\setminus\rho' B$, $k$ is of form $k_h$ (cf. (3.1)), with $h$ a smooth metric on
$F|(B\setminus\rho' B)$.
\end{thm}

The proof needs some preparation.

\begin{lem} % 4.2
Given positively curved metrics $\sk_1,\sk_2,\sk_3$ on $O_{\bP V}(1)$ such that $\sk_1>\sk_2>\sk_3$ away from the 
zero section, there are positive numbers $r_1<r_2<r_3<1$ and a positively curved metric $k$ on $O_{\bP F}(1)$ 
such that for any $j=1,2,3$ 
\[k|O_{\bP F_s}(1)=\sk_j\qquad\text {if}\quad s\in B,\,\, ||s||=r_j.
\]
Moreover, if $||s||$ is sufficiently
close to $r_j$, then $k|O_{\bP F_s}(1)=c_s \sk_j$, $c_s>0$ a constant.
\end{lem}

In the above formulation we took the liberty of identifying $F_s=\{s\}\times  V$ with $ V$ and so
$O_{\bP F_s}(1)$ with $O_{\bP V}(1)$. 

\begin{proof}
The dual metrics $\sk_1^*<\sk_2^*<\sk_3^*$ on $O_{\bP V}(-1)$ are negatively curved, i.e., smooth and 
strongly plurisubharmonic 
away from the zero section. For $j=2,3$ let $q_j=\max \sk_{j-1}^*/\sk_j^*<1$ (maximum taken over 
$O_{ \bP V}(-1)\setminus\bP V$), and choose $p_j>0$ so that
\begin{equation} %4.1
p_j\sk_j^*+q_j<1\qquad\text{when }\sk_1^*\le 1.
\end{equation}
Let
\begin{equation} %4.2
u_1=\sk_1^*,\qquad u_2=p_2\sk_2^*+q_2,\quad\text{and}\quad u_3=p_2(p_3\sk_3^*+q_3)+q_2,
\end{equation}
functions strongly plurisubharmonic on $O_{\bP V}(-1)\setminus\bP V$. Then
\[
u_1<u_2\quad\text{when }\sk_2^*\le 1,\qquad u_1>u_2\quad\text{when } \sk_1^*\ge 1.
\]
The former follows because  $u_1=\sk_1^*\le q_2\sk_2^*$, while away from the zero section $u_2>q_2$; and
on the zero section $u_1=0$, $u_2=q_2$. As to the latter, if $\sk_1^*(\xi)=1$, (4.1), (4.2) imply
$u_2(\xi)<1=u_1(\xi)$. Further, $u_1(t\xi)-u_2(t\xi)$ is a function of $t\in\bR$ of form $at^2+b$, with
\[
b=u_1(0\cdot\xi)-u_2(0\cdot\xi)<0\quad\text{ and }\quad a+b=u_1(\xi)-u_2(\xi)>0;
\]
whence $a>0$ and $at^2+b>0$ if $|t|\ge 1$. Similarly,
\[
u_2<u_3\quad\text{when } \sk_3^*\le 1,\qquad u_2>u_3\quad\text{when }\sk_2^*\ge 1.
\]
In fact, there is an $\var\in(0,1)$ such that
\begin{equation} %4.3
u_2\begin{cases}>u_1+2\var &\text{when }\sk_2^*\le 1+\var\\<u_1-2\var &\text{when }1-\var\le\sk_1^*\le2,\end{cases}
\qquad u_3\begin{cases}>u_2+2\var &\text{when } \sk_3^*\le 1+\var\\
<u_2-2\var &\text{when }\sk_2^*\ge1-\var,\sk_1^*\le 2.\end{cases}
\end{equation}
In particular, 
$\max(u_1,u_2,u_3)=u_j$ when $1-\var\le\sk_j^*\le 1+\var$.

We now bring in a regularized maximum function $M=M_{\var,\var,\var}:\bR^3\to\bR$. This is a smooth, convex
function, increasing in all three variables, $M(a+d,b+d,c+d)=M(a,b,c)+d$, and
\[
M(a,b,c)\ge\max(a,b,c),
\]
with equality if $|a-b|,|b-c|,|a-c|>2\var$, see \cite[I.5.18]{D12}. It follows that 
$u=M(u_1,u_2,u_3):O_{\bP V}(-1)\to[0,\infty)$ is smooth, and strongly plurisubharmonic away from $\bP V$; 
$u(\lambda\xi)=u(\xi)$ if $\lambda\in\bC$ is unimodular; and $u=u_j$ when $1-\var\le\sk_j^*\le1+\var$ by (4.3). Finally,
if $\xi\in O_{\bP V}(-1)\setminus\bP V$, there is a $\delta>0$ such that $u_j(t\xi)-\delta t^2$ is an increasing function of
$t\in(0,\infty)$, for all $j$. Hence 
\[
u(t\xi)-\delta t^2=M\big(u_1(t\xi)-\delta t^2,u_2(t\xi)-\delta t^2, u_3(t\xi)-\delta t^2\big)
\]
also increases and so $du(t\xi)/dt>0$. This implies on the one hand that $u(t\xi)$ increases with $t$, 
whence $u$ attains its minimum on the zero section,
\[
\min u=q=p_2q_3+q_2;
\]
on the other hand that $\xi$ is a regular point of $u$, and
\[
N=\big\{(s,\xi)\in B\times O_{\bP V}(-1): ||s||^2+u(\xi)<1+q\big\}
\]
is a smoothly bounded strongly pseudoconvex neighborhood of $B\times \bP V\subset B\times O_{\bP V}(-1)$. Also, if
$(s,\xi)\in N$ and $\lambda\in\bC$, $|\lambda|\le 1$, then $(s,\lambda\xi)\in N$. All this implies that there is
a negatively curved metric $\kappa$ on the line bundle $O_{\bP F}(-1)=B\times O_{\bP V}(-1)\to B\times\bP V$ whose 
unit disc bundle is $N$. Thus $\kappa(s,\xi)=1$ if and only if $||s||^2+u(\xi)=1+q$.

Let $r_j>0$ satisfy 
\[
r_1^2=q,\qquad r_2^2=1+q-(p_2+q_2),\qquad r_3^2=1+q-p_2(p_3+q_3)-q_2=1-p_2p_3.
\]
(4.1) implies that $p_j+q_j<1$, and so $r_1<r_2<r_3<1$. We claim that if $s\in B$ and $||s||$ is close to
$r_1,r_2$, respectively, $r_3$, then
\begin{equation} %4.4
\kappa(s,\cdot)=\frac1{1+q-||s||^2}\sk_1^*,\qquad\
\frac{p_2}{1+p_2q_3-||s||^2}\sk_2^*,\qquad\text{resp.\quad}
\frac{p_2p_3}{1-||s||^2}\sk_3^*.
\end{equation}
Note that when $||s||$ is exactly equal to $r_j$, the coefficient of $\sk_j^*$ above is $1$.

We verify the last formula in (4.4), the other two follow similarly. Suppose $||s||$ is close to $r_3$, and take
a $\xi$ with $\sk_3^*(\xi)=(1-||s||^2)/p_2p_3$, which  quantity is close to $1$. Hence
\[
||s||^2+u(\xi)=||s||^2+u_3(\xi)=||s||^2+p_2\Big(\frac{1-||s||^2}{p_2}+q_3\Big)+q_2=1+q,
\]
so $\kappa(s,\xi)=1$. Since $\kappa(s,\cdot)$ and 
$\sk_3^*$ are absolutely $2$-homogeneous, this implies
\[
\kappa(s,\xi)=\frac{p_2p_3}{1-||s||^2}\sk_3^*(\xi)\qquad\text{for all }\xi.
\]
Thus the metric $k=\kappa^*$ on $O_{\bP V}(1)$, dual to $\kappa$, is the metric sought.
\end{proof}

\begin{lem} %4.3
Suppose $0\le\rho_1<\rho_2<\rho$, and a metric $h$ on the trivial bundle $F'=\rho B\times V\to \rho B$ has 
semipositive curvature.  Assume that
\[
h(s,\cdot)=h(s',\cdot) \qquad\text{whenever}\quad ||s||=||s'||=\rho_2.
\]
Then $h(s,\cdot)\ge h(s',\cdot)$ whenever $||s||=\rho_1$, $||s'||=\rho_2$.
\end{lem}
 \begin{proof}
 The dual metric $h^*$ on ${F'}^*=\rho B\times V^*\to \rho B$ is seminegatively curved, hence $h^*(\cdot,\xi)$ is 
 a plurisubharmonic
 function on $\rho B$ for any $\xi\in V^*$. By the maximum principle, if $||s||=\rho_1$,
 \[
 h^*(s,\xi)\le\max_{||\sigma||=\rho_2}h^*(\sigma,\xi)=h^*(s',\xi)
 \]
 with any $s'$ of norm $\rho_2$. This implies $h(s,\cdot)\ge h(s',\cdot)$ for the predual metrics.
 \end{proof}
 
 \begin{proof}[Proof of Theorem 4.1]
 We will argue when $V=\bC^2$. Consider the function $\sff_\infty:\bC^2\to[0,\infty)$,
 \[
 \sff_\infty(v)=\max(|v_1|^2,|v_2|^2),\qquad v=(v_1,v_2)\in\bC^2.
 \]
 If $\phi\in\GL(\bC^2)$ is given by\footnote{What matters is that $\phi$ is a contraction in the $l^\infty$
 norm on $\bC^2$, but not in the $l^2$ norm.}
 \[
 \phi(v)=\Big(\frac{3v_1}4+\frac{v_2}5,\frac{3v_1}4-\frac{v_2}5\Big),
 \]
 then $\sff_\infty\circ\phi<\sff_\infty$ on $\bC^2\setminus\{0\}$. Approximate $\sff_\infty$ by an $\sff:\bC^2\to[0,\infty)$ 
that is convex and absolutely
$2$-homogeneous, on $\bC^2\setminus\{0\}$ is smooth, positive, and strongly convex. The approximation should be 
so accurate that $\sff\circ\phi<\sff$ on $\bC^2\setminus\{0\}$. Consider $\psi,\psi_{\sigma\tau}\in\GL(\bC^2)$,
\[
\psi(v_1,v_2)=(v_2,v_1),\qquad\psi_{\sigma\tau}(v_1,v_2)=(e^{i\sigma}v_1,e^{i\tau}v_2),\quad\sigma,\tau\in\bR.
\]
Since $\sff_\infty$ is invariant under these maps, by averaging we can achieve that $\sff$ is invariant, too. Choose
$\var>0$ so that $\sff_2(v)=\var\big(|v_1|^2+|v_2|^2\big)<\sff\big(\phi(v)\big)$ when $v\neq0$.

Let $\sk_j\in\sK_{\bC^2}$, $j=1,2,3$, be given by
\[
\sk_1=\sk_\sff,\qquad\sk_2=\sk_{\sff\circ\phi},\qquad\sk_3=\sk_{f_2}
\]
(cf. Proposition 3.2) and with an equivariant $\Psi:\sK_{\bC^2}\to\sH_{\bC^2}$ let $\sh_j=\Psi(\sk_j)$. 
Thus $\sk_1>\sk_2>\sk_3$. Choose  $r_1<r_2<r_3<1$ and
a positively curved metric $k$ on $O_{\bP F}(1)$  as in Lemma 4.2. 

To understand the curvature of $\Psi k$, note that according to (3.3), 
$\psi_\sK\sk_1=\sk_{\sff\circ\psi}=\sk_\sff=\sk_1$; and
$\sk_1$ is similarly invariant under $(\psi_{\sigma\tau})_\sK$. 
%The same applies to $\sk_3$ (but not to $\sk_2$). 
By equivariance, $\sh_1$ is also invariant under
the corresponding transformations: $\sh_1\circ\psi=\sh_1\circ\psi_{\sigma\tau}=\sh_1$. This means that
in $\sh_1(v)=a|v_1|^2+\re bv_1\overline{v_2}+c|v_2|^2$ the coefficients must satisfy $a=c$, $b=0$,
and so $\sh_1(v)=a\big(|v_1|^2+|v_2|^2\big)$.

Again, by equivariance and by (3.3)
\[
\sh_1\circ\phi=\phi_\sH\big(\Psi (\sk_1)\big)=\Psi(\phi_\sK\sk_\sff)=\Psi(\sk_{\sff\circ\phi})=\sh_2.
\]
Hence $h=\Psi k$ cannot have semipositive curvature over $r_3B$, for if it had, Lemma 4.3 would 
imply $\sh_1\ge\sh_2$,
whereas $\sh_1(v_1,0)=a|v_1|^2$ and
\[
\sh_2(v_1,0)=\sh_1\big(\phi(v_1,0)\big)=\sh_1\Big(\frac{3v_1}4,\frac{3v_1}4\Big)=\frac{18}{16}a|v_1|^2.
\]

This takes care of (a). As to (b),
if $s$ is close to the sphere $\partial(r_3 B)$, then $k|O_{\bP F_s}(1)$ is a constant multiple of $\sk_3$, 
whence of form $\sk_\sh$ with the hermitian metric $\sh$ a multiple of $\sff_2$. 
Therefore taking $R>r_3$ sufficiently close to $r_3$ and replacing $k$ by its pullback along
\[
\bP F=B\times\bP\bC^2\ni(s,x)\mapsto (Rs,x)\in\bP F,
\]
we obtain a new $k$ that, with suitable $\rho,\rho'$, satisfies both (a) and (b).
\end{proof}

\section{Proof of Theorem 2.1} %section 5

We will construct a bundle $E\to S$, a positively curved
metric $k$ on $O_{\bP E}(1)$, and an embedding $\var:B\to S$ so that 
the pullback metric $\var^*k$ on $O_{\bP(\var^*E)}(1)$ is isometric to $k$ of Theorem 4.1. Here and below
$B\subset\bC^m$ is the unit ball, $m=\dim S$.

\begin{lem} %5.1
Let $\Lambda\to S$ be a holomorphic line bundle that admits a positively curved metric, and let 
$0<\sigma<\sigma'<1$. %Let $\dim S=m$ and $B\subset\bC^m$ the unit ball. 
There are a semipositively curved smooth metric $h$ on 
$\Lambda$,  a holomorphic embedding $\var:B\to S$, and a holomorphic section $l$ of $\Lambda|\var(B)$ such 
that $h\circ l=1$ on $\var(\sigma B)$, and $h$ is positively curved on $S\setminus\var(\sigma'B)$.
\end{lem}

\begin{proof} On some open $U\subset S$ choose local coordinates $(s_1,s_2,\dots,s_m):U\to\bC^m$, 
whose origin
$0\in U$. Identifying $s\in U$ with $\big(s_1(s),\dots,s_m(s)\big)\in\bC^m$ we view $U$ as a neighborhood of 
$0\in\bC^m$. If $U$ is sufficiently small, $\Lambda|U$ admits a nowhere vanishing holomorphic section $l_1$.

Let $h_1$ be a positively curved metric on $\Lambda$. Then $u=-\log h_1\circ l_1$ is strongly 
plurisubharmonic
in $U$. Its Taylor approximation can be written, with a linear form $L$, a quadratic form $Q$, and a hermitian
metric $H$ on $\bC^m$, as
\[
u(s)=u(0)+\re\big(L(s)+Q(s)\big)+H(s)+O(||s||^3)\qquad\text{as } s\in U\subset\bC^m\text{ tends to } 0.
\]
Let $l_2(s)= e^{\left(u(0)+L(s)+Q(s)\right)/2}l_1(s)$. This is a holomorphic section of $\Lambda|U$, and
\[
v(s)=-\log h_1\big(l_2(s)\big)=H(s)+O(||s||^3).
\]
%Thus $v-H$ and $d(v-H)$ vanish at $0$.

With a smooth $\chi:\bC^m\to[0,1]$ supported in $\{s: H(s)<{\sigma'}^2\}$, such that $\chi=1$ in a neighborhood of
$\{s:H(s)\le\sigma^2\}$, and with a positive number $t$, let $\chi_t(s)=\chi(s/t)$ and
\[
v_t=\chi_tH+(1-\chi_t)v\in C^\infty(U).
\]
We compute
\begin{equation} %5.1
\begin{aligned}
i\partial\bar\partial v_t=&\chi_t i\partial\bar\partial H+(1-\chi_t) i\partial\bar\partial v+i\partial\chi_t
\wedge \bar\partial(H-v)\\
&+i\partial(H-v)\wedge\bar\partial\chi_t+i(H-v)\partial\bar\partial \chi_t.
\end{aligned}
\end{equation}

Both $i\partial\bar\partial H$ and $i\partial\bar\partial v$ are (strictly) positive. If $\omega>0$ is a real $(1,1)$ form
such that $i\partial\bar\partial H,i\partial\bar\partial v\ge\omega$, then the sum of the first two terms on the right of
(5.1) is $\ge\omega$. On the set where $H(s)\ge (t\sigma')^2$ the remaining terms are $0$. If
$H(s)<(t\sigma')^2$ then $H(s)-v(s)=O(||s||^3)=O(t^3)$ and $d\big(H(s)-v(s)\big)=O(t^2)$; while 
$d\chi_t=O(1/t)$ and $\partial\bar\partial \chi_t=O(1/t^2)$. It follows that the last three terms in (5.1) tend to
$0$ as $t\to 0$. We conclude that if $t>0$ is small, $v_t$ is strongly plurisubharmonic. 
Fix such a $t$, which additionally
has the property that $\{s\in\bC^m:H(s)<t^2\}\subset U$. Fix also a positive $\delta<t^2({\sigma'}^2-\sigma^2)/6$ so that
\begin{equation*}
v_t(s)=\begin{cases} H(s) &\text{when } H(s)\le (t\sigma)^2+5\delta\\
v(s) &\text{when } H(s)\ge(t\sigma')^2-\delta.\end{cases}
\end{equation*}
It follows that
\begin{equation*}
(t\sigma)^2+2\delta>v_t+2\delta\quad\text{if }H< (t\sigma)^2,\qquad (t\sigma)^2+2\delta<v_t-2\delta\quad\text{if }
4\delta< H-(t\sigma)^2< 5\delta.
\end{equation*}

As in the proof of Lemma 4.2 we take a regularized maximum function, but of two variables:
$M=M_{\delta,\delta}:\bR^2\to\bR$ (see \cite[I.5.18]{D12}), and define
\begin{equation*}
w=\begin{cases} M\big(v_t,(t\sigma)^2+2\delta\big)&\text{if } s\in U,\,\,H(s)<(t\sigma)^2+5\delta\\
v_t&\text{if }s\in U,\,\, H(s)\ge(t\sigma)^2+5\delta.\end{cases}
\end{equation*}
Since $M\big(v_t,(t\sigma)^2+2\delta\big)=v_t$ when $4\delta< H-(t\sigma)^2< 5\delta$, our $w$ is
smooth on $U$. The metric $h$ on $\Lambda$ defined by
\[
h=\begin{cases} h_1&\text{over } S\setminus\{s\in U:H(s)<(t\sigma')^2\}\\
e^{v-w}h_1&\text{over }\{s\in U:H(s)<(t\sigma')^2\}\end{cases}
\]
is also smooth, since $v=v_t=w$ in a neighborhood of the set $\{s\in U:H(s)=(t\sigma')^2\}$. It is positively curved 
where it is defined as $h_1$, and it is semipositively curved on $\{s\in U:H(s)<(t\sigma')^2\}$, because there
\[
-\log h\circ l_2=w-v-\log h_1\circ l_2=w=M\big(v_t,(t\sigma)^2+2\delta\big)\text{ or } v_t,
\]
and both $v_t$ and the regularized maximum of two plurisubharmonic functions are plurisubharmonic. 

Where $H<(t\sigma)^2$, $w=(t\sigma)^2+2\delta$ and
$\log h\circ l_2=v-w+\log h_1\circ l_2=-(t\sigma)^2-2\delta$. Hence
\[
h\circ l=1 \qquad\text{with}\quad l=e^{(t\sigma)^2/2+\delta}l_2.
\]
Finally, there is a linear map $\bC^m\to\bC^m$ whose restriction $\var$ to $B\subset\bC^m$ sends $B$ to
$\{s\in U:H(s)<t^2\}\subset S$. The metric $h$, the embedding $\var$, and the section $l$ 
constructed then satisfy the requirements of the lemma.
\end{proof}

\begin{proof}[Proof of Theorem 2.1]
In the proof we will avail ourselves of the following nonstandard language. Given a holomorphic
vector bundle $F\to T$ and $U\subset T$, we will say the part of $O_{\bP F}(1)$ ``over $U$'' to mean
$O_{\bP(F|U)}(1)$. So  ``over'' refers not to the bundle projection of $O_{\bP F}(1)$, but to the
composition $O_{\bP F}(1)\to\bP F\to T$.

Given $r=2,3,\dots$, consider the trivial bundle $F=B\times\bC^r\to B$, and on $O_{\bP F}(1)$ construct
a positively curved metric as $k$ in Theorem 4.1---but now we will denote 
this metric $k_1$. There are $0<\rho<\rho'<1$ such 
that with any $\GL(\bC^r)$ equivariant map $\Psi:\sK_{\bC^r}\to\sH_{\bC^r}$ the metric $\Psi k_1$ on
$F|\rho B$ fails to have semipositive curvature; and there is a smooth metric $h_1$ on $F|(B\setminus\rho'B)$
such that $k_1=k_{h_1}$. Proposition 3.3 implies that $h_1$ is positively curved.

Our projective algebraic manifold $S$ supports a positive line bundle $\Lambda$. Choose real numbers
$\sigma,\sigma'$ so that $\rho'<\sigma<\sigma'<1$, and construct a semipositively curved smooth metric $h$
on $\Lambda$, an embedding $\var:B\to S$, and a holomorphic section $l$ of $\Lambda|\var(B)$ as in
Lemma 5.1. We identify $B$ with its image $\var(B)$, so that we view $B$ as an open subset of $S$.

Extend $F\to B$ to the trivial vector bundle $S\times\bC^r\to S$, that we will keep denoting $F$.
Choose a smooth metric $h_2$ on this extended $F$ that agrees with $h_1$ over $\sigma'B\setminus\rho'B$.
As the curvature of $h$ over $S\setminus\sigma'B$ is positive, there is a choice of $p\in\bN$ such that
the metric $h_3=h^{\otimes p}\otimes h_2$ on $\Lambda^p\otimes F$ has positive curvature over
$S\setminus\sigma'B$. Let $E=\Lambda^p\otimes F$. The section $l$ of $\Lambda|B$ induces an 
isomorphism $\phi: F|B\to E|B$ by setting \(\phi(\xi)=l(s)^{\otimes p}\otimes\xi\) if $\xi\in F_s$.
We compute $h_3\circ\phi$:
\[
h_3\big(\phi(\xi)\big)=h_3(l(s)^{\otimes p}\otimes \xi\big)=h\big(l(s)\big)^p h_2(\xi),\qquad s\in B,\, \xi\in F_s.
\]
Over $\sigma B\setminus\rho'B$ this gives $h_3\circ\phi=h_2=h_1$, and so
$k_1=k_{h_1}=k_{h_3\circ\phi}=\phi_\sK k_{h_3}$ by (3.3). Therefore
\[
k=\begin{cases}\phi_\sK^{-1}k_1&\text{over }\sigma B\\ k_{h_3}&\text{over } S\setminus\rho'B\end{cases}
\]
defines a smooth metric on $O_{\bP E}(1)$.

Now $\big(O_{\bP E}(1),k\big)$ and $\big(O_{\bP F}(1),k_1\big)$ are isometrically
 isomorphic over $\sigma B$. In view of what we know about $k_1$, first this means that  for no equivariant  
$\Psi:\sK_{\bC^r}\to\sH_{\bC^r}$ has the metric $\Psi k$ semipositive curvature. Second, it means that
$k$ has positive curvature over $\sigma B$.
Over $S\setminus\sigma B$, $h_3=h^{\otimes p}\otimes h_2$ is
positively curved: over $S\setminus \sigma'B$ by the choice of $p$, over $\sigma'B\setminus\sigma B$ because 
$h_3=h^{\otimes p}\otimes h_1$, and $h$ is semipositively, while $h_1$ is positively curved there. By Proposition
 3.3 $k=k_{h_3}$ is positively curved over $S\setminus\sigma B$ as well. In sum, $k$ fulfills the
 requirements of the theorem.
\end{proof}

\end{document}